\documentclass[12pt]{amsart}
\usepackage{graphicx}
\usepackage{amssymb, amsmath, amsthm, amsfonts, enumerate}
\usepackage{amsmath,setspace,scalefnt}
\usepackage{verbatim}
\usepackage{multirow}
\usepackage{mathtools}
\usepackage{float}
\usepackage{enumitem}
\restylefloat{figure}
\usepackage[margin=3.5cm]{geometry}

\usepackage{fancyhdr}
\usepackage{hyperref}

\begin{document}

\def\COMMENT#1{}

\newcommand{\case}[1]{\medskip\noindent{\bf Case #1} }
\newcommand{\step}[1]{\medskip\noindent{\bf Step #1} }
\newtheorem{problem}{Problem}
\newtheorem{theorem}{Theorem}
\newtheorem{lemma}[theorem]{Lemma}
\newtheorem{proposition}[theorem]{Proposition}
\newtheorem{corollary}[theorem]{Corollary}
\newtheorem{conjecture}[theorem]{Conjecture}
\newtheorem{claim}[theorem]{Claim}
\newtheorem{definition}[theorem]{Definition}
\newtheorem*{definition*}{Definition}
\newtheorem{fact}[theorem]{Fact}
\newtheorem{observation}[theorem]{Observation}
\newtheorem{question}[theorem]{Question}
\newtheorem{remark}[theorem]{Remark}

\numberwithin{equation}{section}
\numberwithin{theorem}{section}

\def\eps{{\varepsilon}}
\renewcommand{\epsilon}{\varepsilon}
\newcommand{\cP}{\mathcal{P}}
\newcommand{\cT}{\mathcal{T}}
\newcommand{\cL}{\mathcal{L}}
\newcommand{\cG}{\mathcal{G}}
\newcommand{\N}{\mathbb{N}}
\newcommand\ex{\ensuremath{\mathrm{ex}}}
\newcommand{\eul}{e}
\newcommand{\pr}{\mathbb{P}}
\newcommand{\phiind}{\phi^{\rm ind}}

\title[Enumerating Multiplicative Sidon sets]{The number of multiplicative Sidon sets of integers}

\author{Hong Liu}
\email{h.liu.9@warwick.ac.uk}
\address{Mathematics Institute, University of Warwick, Coventry, CV4 7AL, UK}

\author{P\'eter P\'al Pach}
\email{ppp@cs.bme.hu}
\address{Department of Computer Science and DIMAP, University of Warwick, Coventry CV4 7AL, UK and Department of Computer Science and Information Theory, Budapest University of Technology and Economics, 1117 Budapest, Magyar tud\'osok
	k\"or\'utja 2., Hungary}

\thanks{H.L. was supported by the Leverhulme Trust Early Career Fellowship~ECF-2016-523.\\
	P.P.P. was partially supported by the National Research, Development and Innovation Office NKFIH (Grant Nr.~PD115978) and the J\'anos Bolyai Research Scholarship of the Hungarian Academy of Sciences; he has also received funding from the European Research Council (ERC) under the European Union’s Horizon 2020 research and innovation programme (grant agreement No 648509). This publication reflects only its author's view; the European Research Council Executive Agency is not responsible for any use that may be made of the information it contains. This work is connected to the scientific program of the   ``Development of quality-oriented and harmonized R+D+I strategy and functional model at BME'' project, supported by the New Hungary Development Plan (Project ID: T\'AMOP-4.2.1/B-09/1/KMR-2010-0002).}

\begin{abstract}
	A set $S$ of natural numbers is multiplicative Sidon if the products of all pairs in $S$ are distinct. Erd\H{o}s in 1938 studied the maximum size of a multiplicative Sidon subset of $\{1,\ldots, n\}$, which was later determined up to the lower order term: $\pi(n)+\Theta(\frac{n^{3/4}}{(\log n)^{3/2}})$. We show that the number of multiplicative Sidon subsets of $\{1,\ldots, n\}$ is $T(n)\cdot 2^{\Theta(\frac{n^{3/4}}{(\log n)^{3/2}})}$ for a certain function $T(n)\approx 2^{1.815\pi(n)}$ which we specify. This is a rare example in which the order of magnitude of the lower order term in the exponent is determined. It resolves the enumeration problem for multiplicative Sidon sets initiated by Cameron and Erd\H{o}s in the 80s.
	
	We also investigate its extension for generalised multiplicative Sidon sets. Denote by $S_k$, $k\ge 2$, the number of multiplicative $k$-Sidon subsets of $\{1,\ldots, n\}$. We show that $S_k(n)=(\beta_k+o(1))^{\pi(n)}$ for some $\beta_k$ we define explicitly. Our proof is elementary.
\end{abstract}

\maketitle

\section{Introduction}
A set $S\subseteq \N$ is a \emph{multiplicative Sidon set} if all the products $xy$ with $x,y\in S$ are distinct. In other words, $S$ does not contain distinct elements satisfying the equation $a_1a_2=b_1b_2$. The notion of multiplicative Sidon set was introduced by Erd\H{o}s~\cite{Erd38} back in 1938, who studied the maximum size of a multiplicative Sidon subset of $[n]:=\{1,\ldots, n\}$, denoted by $s(n)$. He gave a construction of a multiplicative Sidon set, showing that $s(n)$ is at least $\pi(n)+c'\frac{n^{3/4}}{(\log n)^{3/2}}$ for some constant $c'>0$, and proved an upper bound $\pi(n)+O(n^{3/4})$. The order of magnitude of the lower order term in $s(n)$ was finally pinned down 31 years later by Erd\H{o}s himself~\cite{Erd69}, showing that, for some constant $c>0$,
\begin{equation}\label{eq-erdos-sidon}
	\pi(n) + c'\frac{n^{3/4}}{(\log n)^{3/2}}\le s(n)\le \pi(n) + c\frac{n^{3/4}}{(\log n)^{3/2}}.
\end{equation}
For more on multiplicative Sidon sets and its extensions, we refer the readers to~\cite{Pach15,PS18} and references therein.

Given now the satisfying answer~\eqref{eq-erdos-sidon} on how large a multiplicative Sidon subset of $[n]$ could be. A natural next step would be to estimate how many multiplicative Sidon sets there are in $[n]$. Indeed, enumerating subsets of $[n]$ satisfying various properties was initiated by Cameron and Erd\H{o}s~\cite{CE88} in the 80s. In particular, denoting $S(n)$ the number of multiplicative Sidon subsets of $[n]$, they determined asymptotically the logarithm of $S(n)$. Considering the accuracy on $s(n)$ given by~\eqref{eq-erdos-sidon}, it is natural to ask for a better estimate of $S(n)$. This is the content of one of our main results, which gives much finer count on $S(n)$ with precisions matching that in~\eqref{eq-erdos-sidon}.


\subsection{Main results}
\begin{theorem}\label{thm-count-2sidon}
	There exists $C>0$ such that the number of multiplicative Sidon subsets in $[n]$ satisfies 
	$$T(n)\cdot 2^{(\sqrt{2}+o(1))\frac{n^{3/4}}{(\log n)^{3/2}}}\le S(n)\le T(n)\cdot 2^{C\frac{n^{3/4}}{(\log n)^{3/2}}},$$
	where 
	$$T(n):=\prod_{p \text{ prime: }n^{2/3}< p\leq n}(\lfloor n/p\rfloor+1).$$
\end{theorem}

Theorem~\ref{thm-count-2sidon} is a rare example of enumeration result in which the correct order of magnitude of the lower order term is given. A more explicit formula for the function $T(n)$ is 
$$T(n)=2^{O(n^{2/3})}\cdot \prod\limits_{i=1}^{n^{1/3}}(1+1/i)^{\pi(n/i)},$$ 
see Section~\ref{sec-T}. A more crude estimate is $T(n)=(2^{\alpha}+o(1))^{\pi(n)}$, where
\begin{equation}\label{eq-alpha}
	\alpha:=\sum\limits_{i=1}^\infty{\frac{1}{i}}\log_2(1+1/i)\approx 1.8146.
\end{equation}

For an integer $k\ge 2$, a set $A\subseteq\N$ is \emph{multiplicative $k$-Sidon} if $A$ does not contain $2k$ distinct elements satisfying the equation $a_1a_2\ldots a_k=b_1b_2\ldots b_k$. The maximum size of a multiplicative $k$-Sidon subset of $[n]$, denoted by $s_k(n)$, is closely related to a problem of Erd\H{o}s, S\'ark\"ozy and S\'os~\cite{ESS95} on product representations of powers of integers. It was shown in~\cite{Pach15} that $s_k(n)$ is asymptotically $\pi(n)$ and $\pi(n)+\pi(n/2)$ when $k$ is even and odd respectively.


Our next result concerns multiplicative $3$-Sidon sets. Denote by $S_k(n)$ the number of multiplicative $k$-Sidon subsets in $[n]$. We show that the limit of $S_3(n)^{1/\pi(n)}$ exists.
\begin{theorem}\label{thm-count-3-sidon}
	The number of multiplicative $3$-Sidon set is 
	$$S_3(n)=(\beta+o(1))^{\pi(n)},$$
	for some $\beta>0$.
	Futhermore, for any $\eps > 0$, there exists $N(\eps)$ such that $\beta$ can be approximated within a factor of $1+\eps$ in $N(\eps)$ steps.
\end{theorem}
In fact, we define $\beta$ explicitly in~\eqref{eq-beta} via a family of so-called \emph{product-free} graphs. Moreover, we present upper and lower estimates for $\beta\approx 5.2$ that are within a ratio of 1.002.

Our methods for enumerating multiplicative (3-)Sidon sets can be extended to determine $S_k(n)$ for all $k\ge 2$.

\begin{theorem}\label{thm-k-sidon}
	Let $\alpha,\beta$ and $\beta^{-}$ be defined as in~\eqref{eq-alpha},~\eqref{eq-beta} and~\eqref{eq-c1} respectively. Then the number of multiplicative $k$-Sidon subsets of $[n]$ is
	$$S_k(n)=\left\{ 
	\begin{array}{ll}
	(2^{\alpha}+o(1))^{\pi(n)}, &  \text{ if $k\ge 4$ is even};\\
	(\beta_k+o(1))^{\pi(n)}, & \text{ if $k\ge 5$ is odd},\end{array} \right.
	$$
	for some $\beta_k>0$. Furthermore,
	$$ \beta\ge \beta_5\ge \beta_7\ge \ldots \ge \beta^{-}\approx 5.2366.$$
\end{theorem}


\subsection{Related results}
The past decade has witnessed rapid development in enumeration problems in combinatorics. In particular, a closely related problem of enumerating \emph{additive} Sidon sets, i.e.~sets with distinct sums of pairs, and its generalisation to the so-called $B_h$-sets was studied by Dellamonica, Kohayakawa, Lee, R\"odl and Samotij~\cite{DKLRS16,DKLRS18,KLRS15}. For more recent results on enumerating sets with additive constraints, see e.g.~\cite{BLSh17, BLShT15, BLShT18, Gre04, Hancock-Staden-Treglown,Sap03,Tran}. Many of these counting results use the theory of hypergraph containers introduced by Balogh, Morris and Samotij~\cite{BMS15}, and independently by Saxton and Thomason~\cite{ST15}. We refer the readers to~\cite{BMS15,ST15} for more literature on enumeration problems on graphs and other settings. 

Roughly speaking, the hypergraph container method works well when the (hyper)graph has ``uniform'' edge distribution. In the arithmetic setting, when we forbid additive structure, the corresponding Cayley (type) graph is relatively regular and therefore has a nice edge distribution. However, when we forbid \emph{multiplicative} structure such as the one in multiplicative Sidon property, the induced Cayley graph is highly irregular, making it difficult to apply the hypergraph container method. For an example of enumerating sets with multiplicative constraints, we refer the readers to~\cite{LPP18,McN18} in which primitive sets, i.e.~sets with no element dividing another, are studied. Our methods for enumerating multiplicative Sidon sets are elementary, though we do use an extension of an idea of Kleitman and Winston~\cite{KW} to determine the lower order term in $S(n)$.

It is worth noting that for sets with additive constraints and enumeration for graphs with various properties, the logarithm of the counts are often asymptotically the same as the corresponding extremal functions, with only two known exceptions: (i) the family of graphs without $6$-cycles and (ii) the family of additive Sidon sets. In constrast, as shown by~\cite{LPP18,McN18} and our result on $S_k(n)$, for enumeration of sets with multiplicative constraints, the logarithm of the counts are strictly larger than the corresponding extremal functions.


\medskip

\noindent\textbf{Organisation of the paper.} Section~\ref{sec-prelim} sets up notation and tools needed for the proofs. In Sections~\ref{sec-thm-2sidon},~\ref{sec-thm-3-sidon}, and~\ref{sec-k-sidon}, we prove Theorems~\ref{thm-count-2sidon},~\ref{thm-count-3-sidon} and~\ref{thm-k-sidon} respectively. Some concluding remarks are given in Section~\ref{sec-conclude}.

\section{Preliminaries}\label{sec-prelim}
In this section, we present the tools that will be used later in the proofs. Throughout the paper, we omit floors and ceilings when they are not essential.

\subsection{Number theoretic tools}
For $n\in \N$, denote by $\Omega(n)$ the number of prime divisors of $n$ with multiplicity. Let 
$$L(k):=\sum\limits_{i=1}^k (-1)^{\Omega(i)}$$ 
be the summatory Liouville-function.

The first lemma we need states that each element in $[n]$ either has a ``large'' prime divisor or is a product of two ``small'' numbers.
\begin{lemma}\cite{Erd38}\label{lem-part-type}
	For each $a\in [n]$, we can write $a=uv$ with $v\le u$ such that one of the following holds:
	\begin{itemize}
		\item either $u$ is a prime and $u\ge n^{2/3}$;
		\item or $v\le u\le n^{2/3}$.
	\end{itemize}
\end{lemma}

The next standard estimate follows from Bruns's method, see e.g.~\cite{Erd69}.
\begin{lemma}\label{lem-prime-sieve}
	There exists $c>0$ such that for any primes $p_1\le \ldots\le p_k\le n$, the number of integers $m\le n$ which are not divisible by any of the $p_i$ is at most $$cn\prod_{i\in[k]}\left(1-\frac{1}{p_i}\right).$$
\end{lemma}

We shall also use the following estimate which follows from Mertens's estimate~\cite{Mer}.
\begin{lemma}\label{lem-prime-sieve-estimate}
	There exists $c_1,c_2>0$ such that
         $$\frac{c_1}{\log n}\le \prod_{p \text{ prime}:~ p<n}\left(1-\frac{1}{p}\right)\le \frac{c_2}{\log n}.$$
\end{lemma}


\subsection{Graph theoretic tools}
To bound the number of multiplicative Sidon sets, we will make use of several results from extremal graph theory on graphs that do not contain any 4-cycles. By classical theorems of Erd\H{o}s, R\'enyi and S\'os~\cite{ERS}, and Reiman~\cite{Rei59}, it is well-known that an $n$-vertex $C_4$-free graph of maximum size\footnote{The size of a graph is the number of edges.} has  $(\frac{1}{2}+o(1))n^{3/2}$ edges. We need an extension of this on the maximum size of an unbalanced bipartite $C_4$-free graphs, due to K\H{o}v\'ari, S\'os and Tur\'an~\cite{KST}.
\begin{theorem}\label{thm-c4-unbalanced}
	For $m\le n$, the maximum size of a bipartite $C_4$-free graph on partite sets of size $m$ and $n$ is at most $mn^{1/2}+n$.
\end{theorem}

The following lemma extends the classical result of Kleitman and Winston~\cite{KW} on counting $C_4$-free graphs to the unbalanced bipartite setting.
\begin{lemma}\label{lem-c4-graphcount-unbalanced}
	Given $m,n$ with 
	\begin{equation}\label{eq-m}
		n^{11/12}(\log n)^{5}\le m\le n,
	\end{equation}
	 the number of $C_4$-free bipartite graphs with partite sets of sizes $m$ and $n$ respectively is at most $2^{O(mn^{1/2})}$.
\end{lemma}
The proof of Lemma~\ref{lem-c4-graphcount-unbalanced} will be presented in Section~\ref{sec-c4-graphcount-unbalanced}. The exponent $O(mn^{1/2})$ is optimal up to the constant factor. It would be interesting to remove the constraints on $m$: is it true that for $m\le n$, the number of bipartite $C_4$-free graphs with partite sets of sizes $m$ and $n$ respectively is at most $2^{O(mn^{1/2}+n)}$. Nonetheless, the above version suffices for our purposes.

We also need the following bound on the maximum size of a $C_6$-free bipartite graph due to Gy\H{o}ri~\cite{Gyo97}.
\begin{theorem}\label{thm-c6-unbalanced}
	For $m\le n$, the maximum size of a bipartite $C_6$-free graph on partite sets of size $m$ and $n$ is at most $m^{2}/2+2n$.
\end{theorem}

\subsection{Estimating the function $T(n)$}\label{sec-T}
For $n\geq 2$, let $p_0=p_0(n)$ be the smallest prime larger than $n^{2/3}$. Note that $p_0\leq n$. Furthermore, let $k_0=k_0(n)=\lfloor n/p_0\rfloor\approx n^{1/3}$. Then
\begin{eqnarray*}
	T(n)&=&\prod_{p \text{ prime: }n^{2/3}< p\leq n}(\lfloor n/p\rfloor+1)\\
	&=&2^{\pi(n)-\pi(\frac{n}{2})}\cdot 3^{\pi(\frac{n}{2})-\pi(\frac{n}{3})}\cdots  k_0(n)^{\pi(\frac{n}{k_0(n)-1})-\pi(\frac{n}{k_0(n)})}\cdot (k_0(n)+1)^{\pi(\frac{n}{k_0(n)})-\pi(p_0(n)-1)}\\
	&=&(k_0(n)+1)^{-\pi(p_0(n)-1)}\prod\limits_{i=1}^{k_0(n)}(1+1/i)^{\pi(\frac{n}{i})}=\prod\limits_{i=1}^{k_0(n)}(1+1/i)^{\pi(\frac{n}{i})-\pi(p_0(n)-1)}.
\end{eqnarray*}
Let 
$$R(n):=\prod\limits_{i=1}^{k_0(n)}(1+1/i)^{\pi(n/i)}.$$
Note that for any $c>1$, by the Prime Number Theorem, $$\prod\limits_{i=1}^{k_0}(1+1/i)^{\pi(p_0(n)-1)}=(k_0+1)^{\pi(p_0(n)-1)}\leq 2^{cn^{2/3}}.$$ 
Thus 
$$R(n)\cdot 2^{-cn^{2/3}}\leq T(n)\leq R(n).$$
\section{Proof of Theorem~\ref{thm-count-2sidon}}\label{sec-thm-2sidon}
\subsection{Lower bound}
We shall construct multiplicative Sidon sets consisting of two parts $A$ and $B$, where each element in $A$ has a prime divisor larger than $n^{2/3}$, while each element in $B$ is a product of two primes less than $n^{1/2}$.

Let $G$ be a $C_4$-free graph of maximum size on vertex set
$$V(G)=\{p: p\leq n^{1/2}, ~p\text{ is a prime}\}.$$
By the Prime Number theorem, $|V(G)|=(2+o(1))\frac{n^{1/2}}{\log n}$; and by the aforementioned result of Reiman~\cite{Rei59},
$$e(G)=\left(\frac{1}{2}+o(1)\right)|V(G)|^{3/2}=(\sqrt{2}+o(1))\frac{n^{3/4}}{(\log n)^{3/2}}.$$ 
Let $B^*\subseteq[n]$ contain exactly those products $pq$ for which $p$ and $q$ are connected by an edge in $G$, i.e.
$$B^*=\{pq:~ pq\in E(G)\}.$$ 
Notice that $B^*$ is a multiplicative Sidon set and $|B^*|=e(G)=(\sqrt{2}+o(1))\frac{n^{3/4}}{(\log n)^{3/2}}$. Indeed, if $B^*$ contains a solution $(p_1q_1)(p_2q_2)=(p_3q_3)(p_4q_4)$ with distinct $p_iq_i$, $i\in[4]$, then as $p_i,q_i$ are primes, the sets $\{p_1,p_2,q_1,q_2\}$ and $\{p_3,p_4,q_3,q_4\}$ are identical, consisting of 4 distinct elements. This, however, would imply that $\{p_1, p_2, q_1, q_2\}$ induces a copy of $C_4$ in $G$, a contradiction.

Observe  that if  each element of a set $A\subseteq[n]$ has a prime divisor, which does not divide any other element of $A\cup B^*$, then $A\cup B^*$ is also a multiplicative Sidon set. To construct such a set $A$, for every prime $p$ larger than $n^{2/3}$, include at most one multiple of $p$ to $A$. For each such large prime $p$, the number of choices is $\lfloor n/p \rfloor +1$. Since these choices are independent, the number of ways to construct $A$ is precisely 
$$\prod_{p \text{ prime: } n^{2/3}<p\le n}\left(\lfloor n/p \rfloor +1\right)=T(n).$$ 
Finally note that, for every $B\subseteq B^*$, the set $A\cup B\subseteq A\cup B^*$ is a multiplicative Sidon set. Therefore, the number of multiplicative Sidon sets is at least 
$$T(n)\cdot 2^{|B^*|}\ge T(n)\cdot 2^{\frac{(\sqrt{2}+o(1))n^{3/4}}{(\log n)^{3/2}}},$$
as desired.

\subsection{Upper bound}
Our strategy of bounding the number of multiplicative Sidon sets is to partition elements into several types according to their largest prime divisors and bound the number of choices for each type using its structural information.

Let $S\subseteq[n]$ be an arbitrary multiplicative Sidon set. We may assume that $S$ does not contain any perfect squares. Indeed, there are at most $\sqrt{n}$ perfect squares in $n$, contributing a negligible factor of $2^{\sqrt{n}}$. The purpose of this is to avoid loops appearing in auxiliary graphs that we will introduce later. Partition the elements of $S$ into the following two types:
\begin{itemize}
	\item[] $A:=\{a\in S: \exists~~ n^{2/3}<p \text{ prime s.t. } p\mid a\}$;
	
	\item[] $B:=\{a\in S: \text{ all prime divisors of }a\text{ are at most }n^{2/3} \}$.
\end{itemize}

We further partition $A$ depending on whether an element has its own distinct large prime divisor:
\begin{itemize}
	\item[] $A_1:=\{a\in A: \exists~ n^{2/3}<p \text{ prime s.t. } p\mid a\text{ but } p\nmid b \text{ for any } b\in S\setminus \{a\}\}$;
	
	\item[] $A_2:=\{a\in A: \exists~ n^{2/3}<p \text{ prime and } \exists~b\in S\setminus\{a\} \text{ s.t. } p\mid (a,b)\}$.
\end{itemize}

By Lemma~\ref{lem-part-type}, we can write each $a\in B$ as $a=uv$ with $v\le u\le n^{2/3}$. We will fix one such representation $u,v$ such that $v$ is minimum, that is,
\begin{equation}\label{eq-min}
	(u,v)=(u_a,v_a): ~ uv=a, ~ v\le u\le n^{2/3}~\text{and } \forall~ u'v'=a, ~v\le\min\{u',v'\}.
\end{equation}
We then further partition $B$ according to the value $v$ in this representation:
\begin{itemize}
	\item[] $B_1:=\{a\in B: v\le n^{1/3}\}$;
	
	\item[] $B_2:=\{a\in B: n^{1/3}< v\le \frac{n^{1/2}}{(\log n)^8}\}$;
	
	\item[] $B_3:=\{a\in B: \frac{n^{1/2}}{(\log n)^8}< v\le n^{1/2}\}$.
\end{itemize}

Clearly, $S=A_1\cup A_2\cup B_1\cup B_2\cup B_3$. In the following subsections, we shall bound from above the number of ways to construct each $A_i$ and $B_i$. As we shall see later, the main term $T(n)$ is given by the set $A_1$. For the sets $A_2,B_1,B_2$, we shall show that each of them can have size $\frac{8n^{3/4}}{(\log n)^4}$. Since the number of such small sets is at most
$$\sum_{i\le \frac{8n^{3/4}}{(\log n)^4}}{n\choose i}\le n^{\frac{8n^{3/4}}{(\log n)^4}}\le 2^{\frac{12n^{3/4}}{(\log n)^3}},$$
the contribution from $A_2,B_1,B_2$ is negligible. At the end, we shall show that the number of choices of $B_3$ corresponds to the lower order term $2^{\Theta\left(\frac{n^{3/4}}{(\log n)^{3/2}}\right)}$.

\subsubsection{Choosing $A_1$}
Recall that each element of $A_1$ is divisible by a prime $p>n^{2/3}$, and $p$ can not divide any other element of $S$. This means that $A_1$ can contain at most one multiple of $p$. Thus, the number of $A_1$ sets is precisely $T(n)$.

\subsubsection{Choosing $A_2$}
Consider now those primes $n^{2/3}<p\le n$ that divide at least two elements of $S$. For each such prime $p$, let $m_p\ge 2$ be the number of multiples of $p$ contained in $S$.
Construct an auxiliary bipartite graph $\Gamma$ on partite sets $X$ and $Y$, where $X$ consists of all primes in $(n^{2/3},n]$ that have at least two multiples in $A_2$, and $Y=[n^{1/3}]$. In $\Gamma$, $p\in X$ and $\ell\in Y$ form an edge if and only if $p\ell\in S$. Note that the degree of $p\in X$ is exactly $m_p\ge 2$. Since $S$ is a multiplicative Sidon set, it is not hard to see that $\Gamma$ is $C_4$-free. From Theorem~\ref{thm-c4-unbalanced}, we only get $|A_2|=e(\Gamma)\le \pi(n)$, which is too large. However, using the fact that $\Gamma$ has minimum degree at least 2 on $X$, we can get a much better bound as follows.

A \emph{hat} in $\Gamma$ is a copy of $P_3$, a $3$-vertex path, with mid-point in $X$.  As $\Gamma$ is $C_4$-free, no two hats share the same pair of endpoints in $Y$, i.e.
$$\sum_{p\in X}\binom{m_p}{2}\leq {|Y|\choose 2}= \binom{n^{1/3}}{2}.$$ 
Therefore, as $m_p\ge 2$, $A_2$ has small size:
$$|A_2|= \sum\limits_{p\in X} m_p\leq 2\sum_{p\in X}\binom{m_p}{2}\leq n^{2/3}.$$ 

\subsubsection{Choosing  $B_1$}
By definition, for every $a\in B_1$, its representation $a=uv$ satisfies $v\le n^{1/3}$ and $u\le n^{2/3}$. Let $\Gamma$ be an auxiliary bipartite graph on vertex sets $U$ and $V$, where $U=[n^{2/3}]$ and $V=[n^{1/3}]$. For $u\in U$ and $v\in V$, $uv\in E(\Gamma)$ if and only if $uv=a$ is the representation for some $a\in B_1$. Similarly, $\Gamma$ is $C_4$-free as $B_1$ is a multiplicative Sidon set. Then by Theorem~\ref{thm-c4-unbalanced}, we see that $B_1$ must be small: 
$$|B_1|=e(\Gamma)\le |V||U|^{1/2}+|U|=2n^{2/3}.$$

\subsubsection{Choosing  $B_2$}
Let $R:=8\log_2\log n$. We further partition $B_2$ into subsets $B_2^1,B_2^2,\ldots$, such that for each $r\ge 1$,
$$B_2^r:=\left\{a\in B_2: \frac{n^{1/2}}{2^{R+r}}< v\le \frac{n^{1/2}}{2^{R+r-1}}=:M_r\right\}.$$
By the definition of $B_2$, $ v>n^{1/3}$, so $B_2$ is partitioned into at most $\log n$ subsets $B_2^r$.
Also notice that for each $a=uv\in B_2^r$, we have
$$u\le 2^{R+r}\cdot n^{1/2}=:N_r\le n^{2/3}.$$

For each set $B_2^r$, associate it with an auxiliary bipartite graph $\Gamma^r$ on partite sets $U:=[N_r]$ and $V:=[M_r]$, such that $uv\in E(\Gamma^r)$ if and only if $uv=a$ is the chosen representation for some $a\in B_2^r$. As before, the fact that  $B_2^r$ is a multiplicative Sidon set implies that $\Gamma^r$ is $C_4$-free. By Theorem~\ref{thm-c4-unbalanced}, we see that $$|B_2^r|=e(\Gamma^r)\le N_r+\sqrt{N_r}M_r\le n^{2/3}+2^{(R+r)/2}\cdot n^{1/4}\cdot \frac{n^{1/2}}{2^{R+r-1}}=n^{2/3}+\frac{2n^{3/4}}{(\log n)^4}\cdot \frac{1}{2^{r/2}}.$$
Therefore, $B_2$ has small size:
$$|B_2|=\sum_{r\le \log n}|B_2^r|\le \frac{8n^{3/4}}{(\log n)^4}.$$

\subsubsection{Choosing  $B_3$}
Set again $R:=8\log_2\log n$. Partition $B_3$ into $R$ subsets $B_3^1,\ldots,B_3^R$ such that for each $r\in [R]$,
$$B_3^r:=\left\{a\in B_3: \frac{n^{1/2}}{2^r}< v\le \frac{n^{1/2}}{2^{r-1}}\right\}.$$

Fix $r\in [R]$ and an arbitrary $a\in B_3^r$ with representation $a=uv$. We claim that $v$ does not have a prime divisor less than $n^{1/7}$. Indeed, suppose $p<n^{1/7}$ is a prime divisor of $v$, then 
$$u\cdot p<2^r\cdot n^{1/2}\cdot n^{1/7}<n^{2/3}$$ 
and the representation $(u\cdot p, v/p)$ contradicts the minimality of $v$ in~\eqref{eq-min}. Thus, the number of choices for $v$, by Lemmas~\ref{lem-prime-sieve} and~\ref{lem-prime-sieve-estimate}, is at most
$$\frac{n^{1/2}}{2^{r-1}}\cdot c\prod_{p \text{ prime: }p<n^{1/7}}\left(1-\frac{1}{p}\right)\le \frac{14cc_2n^{1/2}}{2^r\cdot \log n}=:M_r.$$
Similarly, $u$ does not have a prime divisor $p\in [2^{2r}, n^{1/7}]$. Suppose there is such $p|u$, then 
$$\frac{u}{p}\le \frac{2^rn^{1/2}}{2^{2r}}\le \frac{n^{1/2}}{2^{r}}<v,$$
and 
$$v\cdot p\le n^{1/2}\cdot n^{1/7}\le n^{2/3}.$$  
Then $(u/p, v\cdot p)$ contradicts the minimality of $v$. We can similarly bound the number of choices for $u$ from above by
$$2^r\cdot n^{1/2}\cdot  c\prod_{p \text{ prime: }2^{2r}<p<n^{1/7}}\left(1-\frac{1}{p}\right)\le 2^r n^{1/2}\cdot \frac{7cc_2\cdot 2r}{c_1\log n}=:N_r.$$

Associate $B_3^r$ with an auxiliary bipartite graph $\Gamma^r$ on partite sets $U$ and $V$ of sizes $N_r$ and $M_r$ respectively in which $uv\in \Gamma^r$ if and only if $(u,v)$ is a representation for some element of $B_3^r$. As $B_3^r$ is a multiplicative Sidon set, $\Gamma^r$ is $C_4$-free. Thus, every choice of $B_3^r$ corresponds to one such bipartite $C_4$-free graph $\Gamma^r$. In other words, the number of choices for $B_3^r$ is at most the number of bipartite $C_4$-free graphs on bipartite sets of sizes $N_r$ and $M_r$ respectively. By Lemma~\ref{lem-c4-graphcount-unbalanced}, we get that the number of choices of $B_3^r$ is $2^{O(M_rN_r^{1/2})}$, where
$$O(M_rN_r^{1/2})= O\left(\frac{n^{1/2}}{2^r\cdot \log n}\cdot \frac{r^{1/2}\cdot 2^{r/2}\cdot n^{1/4}}{(\log n)^{1/2}}\right)= O\left(\frac{r^{1/2}}{2^{r/2}}\cdot \frac{n^{3/4}}{(\log n)^{3/2}}\right).$$
As $\sum_{r\ge 1}\frac{r^{1/2}}{2^{r/2}}$ converges, we conclude that the number of choices for $B_3$ is at most
$$2^{O(\sum_{r\in [R]}(M_rN_r^{1/2}))}=2^{O\left(\frac{n^{3/4}}{(\log n)^{3/2}}\right)}.$$
To finish the proof of Theorem~\ref{thm-count-2sidon}, it remains to prove Lemma~\ref{lem-c4-graphcount-unbalanced}.

\subsection{Unbalanced bipartite $C_4$-free graphs}\label{sec-c4-graphcount-unbalanced}
The proof of Lemma~\ref{lem-c4-graphcount-unbalanced} builds on the idea of Kleitman and Winston~\cite{KW}. We need two of their lemmas. The first one is the graph container lemma, which bounds the number of independent sets in graphs with relatively uniform edge distribution.
\begin{lemma}\label{lem-graph-container}
	Let $n,q$ be integers, and $R$ and $\beta\in[0,1]$ be reals satisfying $R\ge e^{-\beta q}n$. Suppose $G$ is an $n$-vertex graph such that for every $U\subseteq V(G)$ with $|U|\ge R$, 
	$$e(G[U])\ge \beta{|U|\choose 2},$$ 
	then for every integer $s\ge q$, the number of independent sets of size $s$ is at most 
	$${n\choose q}{R\choose s-q}.$$
\end{lemma}
The second lemma is a key step for bounding the number of $C_4$-free graphs.
\begin{lemma}\label{lem-KW-build-C4free}
	There exists $K>1$ such that the following holds. Let $G$ be an $n$-vertex $C_4$-free graph with $\delta(G)\ge d-1$. Then the number of ways to build a $C_4$-free graph $G'$ by adding a vertex of degree $d$ to $G$ is at most 
	$$2^{Kn^{1/2}}.$$
\end{lemma}
\begin{proof}[Proof of Lemma~\ref{lem-c4-graphcount-unbalanced}]
	Let $G$ be a bipartite $C_4$-free graph with partite sets $U$ and $V$ of sizes $n$ and $m$ respectively. Let $w_{n+m},\ldots, w_1$ be a \emph{minimum degree ordering} of $V(G)$, that is, for each $i\in \{n+m,\ldots, 1\}$, $w_i$ is a vertex of minimum degree in $G_i:=G\setminus\{w_{n+m},\ldots, w_{i+1}\}$. By definition, $G_i\subseteq G$ is $C_4$-free and 
	$$d_{G_i}(w_i)\le d_{G_{i-1}}(w_{i-1})+1=\delta(G_{i-1})+1< i.$$
	Reversing this process, we see that every $C_4$-free bipartite graph on partite sets $U$ and $V$ of sizes $n$ and $m$ can be obtained as follows:
	\begin{itemize}
		\item[(S1)] choose an ordering $w_{n+m},\ldots, w_1$ and for each $i\in [n+m]$, choose $d_i< i$ and decide whether $w_i$ belongs to $U$ or $V$;
		\item[(S2)] let $G_1$ be the 1-vertex graph on vertex set $\{w_1\}$ and for each $i\in \{2,\ldots, n+m\}$, add a vertex $w_i$ to $G_{i-1}$ such that
		\begin{itemize}
			\item $G_i:=G_{i-1}\cup \{w_i\}$ is $C_4$-free; and
			\item $d_{G_i}(w_i)=d_i\le \delta(G_{i-1})+1$.
		\end{itemize} 
	\end{itemize}
	Note that, using the bounds on $m$, i.e.~\eqref{eq-m}, the number of choices for (S1) is at most
	$$(n+m)!\cdot (n+m)!\cdot 2^{n+m}\le 2^{5n\log n}=2^{o(mn^{1/2})},$$
	which is negligible. 
	
	Let $s_i$ be the number of choices for $(G_i,w_i)$, the $i$-th step of (S2), i.e.~the number of ways to add $w_i$. Let $U_{i-1}\subseteq U$ and $V_{i-1}\subseteq V$ be the partite sets of $G_{i-1}$, and $a_{i-1}:=|U_{i-1}|$, $b_{i-1}:=|V_{i-1}|$. Note that both $\{a_i\}_i$ and $\{b_i\}_i$ are non-decreasing integer sequences. It suffices to show that $$\prod_{i\in[n+m]}s_i\le 2^{520Kmn^{1/2}},$$
	where $K$ is the constant from Lemma~\ref{lem-KW-build-C4free}.
	
	Note first that the total contribution from all vertices in $V$ is easy to bound. Indeed, for each vertex $w_i\in V$, by Lemma~\ref{lem-KW-build-C4free}, the corresponding $s_i$ satisfies $s_i\le 2^{K(n+m)^{1/2}}$. Thus, using that $|V|=m$, the total contribution from vertices in $V$ is at most 
	\begin{equation}\label{eq-contr-V}
		\prod_{i:~w_i\in V}s_i=(2^{K(n+m)^{1/2}})^{|V|}\le 2^{2Kmn^{1/2}}.
	\end{equation}
	
	We now turn to vertices in $U$. Suppose that $a_{i-1}\le 60m$. As $b_{i-1}\le |V|=m$, by Lemma~\ref{lem-KW-build-C4free}, we see that $$s_i\le 2^{K(a_{i-1}+b_{i-1})^{1/2}}\le 2^{8Km^{1/2}}.$$ 
	Let $i_0$ be the maximum index such that $a_{i_0-1}\le 60m$. Note that $i_0=a_{i_0-1}+b_{i_0-1}+1\le 61m+1$. Thus the total contribution up to the $i_0$-th step is at most
		\begin{equation}\label{eq-contr-small-a}
	\prod_{i\le i_0}s_i=(2^{8Km^{1/2}})^{61m+1}\le 2^{500Kmn^{1/2}}.
	\end{equation}
	
	Consider now $i$-th steps with $i>i_0$, so $a_{i-1}>60m$. We say that the $i$-th step $(G_i,w_i)$ is \emph{balanced} if 
	\begin{itemize}
		\item[(B1)] $b_{i-1}\ge a_{i-1}^{5/6}(\log a_{i-1})^2$; and
		\item[(B2)] $d_i\ge \frac{b_{i-1}}{a_{i-1}^{1/2}\log b_{i-1}}$.
	\end{itemize}
    We call a step \emph{biased} if it is not balanced. We shall bound the contribution from biased and balanced steps separately. For (notational) brevity, write $a:=a_{i-1}$, $b:=b_{i-1}$, $w:=w_i$, $d:=d_i$, $A:=U_{i-1}$, $B:=V_{i-1}$ and $H:=G_{i-1}$.
    
    Note that, as $a>60m\ge 60b$ and $n\ge m$,
    \begin{equation}\label{eq-lower-a}
    a^{1/2}\ge 7m^{1/2}\ge\frac{7m}{n^{1/2}}.
    \end{equation}
    By Theorem~\ref{thm-c4-unbalanced}, $e(H)\le 2a^{1/2}b$, and so 
    \begin{equation}\label{eq-ave-deg-H-2}
    d\le \delta(H)+1\le \frac{2e(H)}{a+b}+1\le \frac{4b}{a^{1/2}}+1.
    \end{equation} 
	\begin{claim}\label{cl-balanced-reduction}
		The contribution from all biased steps $(G_i,w_i)$ is at most $2^{4mn^{1/2}}$.
	\end{claim}
	\begin{proof}
		Suppose first that $d\le \frac{m}{n^{1/2}\log b}$. Recall that we are adding $w$ to $A\subseteq U$. So we have $s_i\le {|B|\choose d}$. Consequently, 
		$$\log s_i\le \log {b\choose d}\le d\log b\le \frac{m}{n^{1/2}}.$$ 
		As $|U|=n$, total contribution from steps with such small $d$ is at most $2^{2mn^{1/2}}$. We may then assume that
		\begin{equation}\label{eq-d}
			d\ge \frac{m}{n^{1/2}\log b}.
		\end{equation}
		
		On the other hand, by~\eqref{eq-ave-deg-H-2}, we have
		\begin{equation}\label{eq-ave-deg-H}
		d\le  \frac{4b}{a^{1/2}}+1\le \frac{5b}{a^{1/2}},
		\end{equation} 
		where in the last inequality we assume $b\ge a^{1/2}$, as otherwise $d\le 5$, contradicting~\eqref{eq-d} and~\eqref{eq-m}. Then~\eqref{eq-d}, together with~\eqref{eq-m},~\eqref{eq-lower-a} and~\eqref{eq-ave-deg-H}, implies that
		$$b\ge \frac{d}{5}\cdot a^{1/2}\ge \frac{m}{5n^{1/2}\log b}\cdot \frac{7m}{n^{1/2}}\ge \frac{m^2}{n\log m}\ge n^{5/6}(\log n)^2\ge a^{5/6}(\log a)^2,$$
		which is~(B1). In other words, we need only consider biased steps violating~(B2), i.e.
		$$d\le \frac{b}{a^{1/2}\log b}.$$ 
		But then we have 
		$$\log s_i\le \log {b\choose d}\le d\log b\le \frac{b}{\sqrt{a}}\le \frac{m}{\sqrt{a}}\le \frac{2m}{\sqrt{a}+\sqrt{a-1}}=2m\left(\sqrt{a}-\sqrt{a-1}\right).$$
		Thus, the total such contribution is at most 
		$$\prod_{a=1}^n 2^{3m\left(\sqrt{a}-\sqrt{a-1}\right)}=2^{3m\sum_{a=1}^n\left(\sqrt{a}-\sqrt{a-1}\right)}=2^{3mn^{1/2}}.$$
		Hence, the total contribution from all biased steps is at most $2^{2mn^{1/2}}+2^{3mn^{1/2}}\le 2^{4mn^{1/2}}$ as claimed.
	\end{proof}
	
	
	We shall now bound the contribution from balanced steps. 
	\begin{claim}\label{cl-balanced-contribution}
		Let $a,b,d$ satisfy (B1) and (B2), and $H$ be a $C_4$-free bipartite graph with partite sets $A$ and $B$ of sizes $a$ and $b$ respectively. Then the number of bipartite $C_4$-free graphs $H'$ obtained from adding a vertex $w$ of degree $d\le \delta(H)+1$ to $A$ is at most $$2^{\frac{13b}{a^{1/2}}\log\frac{a}{b}}.$$
	\end{claim}
	
	We first show how the claim implies the desired bound. From the claim, we see that for the balanced step $(G_i,w_i)$, 
	$$s_i\le 2^{\frac{13b}{a^{1/2}}\log\frac{a}{b}}\le 2^{13m^{1/2}\left(\frac{b}{a}\right)^{1/2}\log\frac{a}{b}}\le 2^{13m^{1/2}\left(\frac{1}{60}\right)^{1/2}\log 60}\le 2^{10m^{1/2}},$$
	as $a/b\ge 60$ and $x^{-1/2}\log x$ is decreasing when $x\ge 60$.
	Thus, the total contribution from all balanced steps to (S2) is at most $2^{10m^{1/2}n}$. This, together with~\eqref{eq-contr-V},~\eqref{eq-contr-small-a} and Claim~\ref{cl-balanced-reduction}, implies that the total number of choices for~(S2) is at most $2^{520Kmn^{1/2}}$ as desired.
	It remains to prove the above claim.
	\begin{proof}[Proof of Claim~\ref{cl-balanced-contribution}]
		Build an auxiliary graph $\Gamma$ on vertex set $B$ in which $uv\in E(\Gamma)$ if and only if $u$ and $v$ have a common neighbour in $H$. Note that to add $w$ to $A$ so that $H'$ is $C_4$-free, the neighbourhood $N_{H'}(w)\subseteq B$ must be an independent set in $\Gamma$. It then suffices to bound the number of independent sets of size $d$ in $\Gamma$, denoted by $i_{\Gamma}(d)$.
		
		Note that for any $Z\subseteq B$,
		\begin{equation}\label{eq-locally-dense}
		\sum_{v\in A}d_H(v,Z)=\sum_{z\in Z}d_H(z)\ge |Z|\delta(H)\ge (d-1)|Z|.
		\end{equation}
		
		Fix 
		$$R=\frac{10a}{d-1},\quad \beta=\frac{(d-1)^2}{2a}, \quad \text{ and }\quad q=\frac{2a\log b}{(d-1)^2}.$$ Then $\beta q= \log b$ and so $R\ge e^{-\beta q}b$ due to~\eqref{eq-ave-deg-H-2} and that $a/b\ge 60$. We will apply Lemma~\ref{lem-graph-container} with $\Gamma$, $b$ and $d$ playing the roles of $G$, $n$ and $s$ respectively. We need to check that $\Gamma$ is locally dense. Fix an arbitrary $Z\subseteq B$ with $|Z|\ge R$. Then as $H$ is $C_4$-free, distinct copies of $P_3$ with mid-points in $A$ correspond to distinct edges in $\Gamma$. By the convexity of the function $f(x)={x\choose 2}$, we have
		\begin{eqnarray*}
			e(\Gamma[Z])\ge  \sum_{v\in A}{d_H(v,Z)\choose 2}\ge a\cdot {\frac{1}{a}\sum_{v\in A}d_H(v,Z)\choose 2}\stackrel{(\ref{eq-locally-dense})}{=}a\cdot {\frac{1}{a}(d-1)|Z|\choose 2}\ge\beta{|Z|\choose 2}.
		\end{eqnarray*}
		Thus by Lemma~\ref{lem-graph-container}, 
		$$\log i_{\Gamma}(d)\le \log {b\choose q}{R\choose d}\le q\log b+d\log\frac{eR}{d}.$$
		For the first term, we have from the balanced-ness that
		$$q\log b=\frac{2a(\log b)^2}{(d-1)^2}\stackrel{(B2)}{\le}\frac{4a^2(\log b)^4}{b^2}\stackrel{(B1)}{\le}\frac{b}{a^{1/2}}.$$
		For the second term, using that $x\log\frac{a^{1/2}}{x}$ is increasing when $x\le 6b/a^{1/2}\le  a^{1/2}/10$, we have
		$$d\log\frac{eR}{d}=d\log\frac{10e\cdot a}{d(d-1)}\le 2d\log\frac{6a^{1/2}}{d}\stackrel{(\ref{eq-ave-deg-H})}{\le} \frac{12b}{a^{1/2}}\cdot \log\frac{a}{b},$$
		as desired.   	   
	\end{proof}
	
	This completes the proof of Lemma~\ref{lem-c4-graphcount-unbalanced}.
\end{proof}

\section{Proof of Theorem~\ref{thm-count-3-sidon}}\label{sec-thm-3-sidon}
We will show in this section that the limit of $S_3(n)^{1/\pi(n)}$ exists as $n$ tends to infinity, and the limit $\beta$ is determined by a family of \emph{product-free} graphs defined below. We further give numerical estimates of $5.2366<\beta< 5.2468$ that are within a ratio of $1.002$.

\begin{definition*}
	A graph $G$ on vertex set $\N$ is \emph{product-free} if any three (not necessarily distinct) edges $a_1b_1,a_2b_2,a_3b_3$ in $G$ satisfy $\frac{a_1}{b_1}\cdot \frac{a_2}{b_2}\neq \frac{a_3}{b_3}$. Denote by $G_k$ the induced subgraph of $G$ on $[k]$, and by $\cG$ the family of all product-free graphs. Define
	\begin{equation}\label{eq-beta}
		\beta:=\sup_{G\in\cG}~\prod\limits_{k=1}^\infty (e(G_k)+k+1)^{\frac{1}{k^2+k}}.
	\end{equation}
	\end{definition*}

We first note that $\beta$ is well-defined. Indeed, for any $G\in\cG$, clearly we have $e(G_k)\le {k\choose 2}$, implying that 
$$\prod\limits_{k=1}^\infty (e(G_k)+k+1)^{\frac{1}{k^2+k}}\le \sqrt{2}\prod_{k=2}^\infty (k^2)^{\frac{1}{k^2}}=\sqrt{2}\exp\left(\sum_{k=2}^\infty \frac{2\log k}{k^2}\right)<10.$$

\subsection{Lower bound}\label{sec-beta-theory-lower}
Let $\varepsilon>0$ be arbitrary and $G\in \cG$ be product-free with 
$$\prod_{k=1}^K (e(G_k)+k+1)^{\frac{1}{k^2+k}}>\beta-\varepsilon/2,$$ 
for some $K=K(\eps)$. Consider subsets $S\subseteq [n]$ constructed as follows. For each prime $p>\sqrt{n}$, include at most two multiples of $p$ in $S$ in such a way that if $S$ contains two multiples of $p$, say $pa$ and $pb$, then $ab\in E(G)$. 

We claim that all these sets satisfy the multiplicative 3-Sidon property. Suppose that $a_1a_2a_3=b_1b_2b_3$ for some distinct $a_1,a_2,a_3,b_1,b_2,b_3\in S$, each of which has a prime factor larger than $\sqrt{n}$. Consequently, the largest prime factors of $a_i,b_i$ must appear on both sides of the equation. Without loss of generality, we may then assume that $a_i=p_ia_i',b_i=p_ib_i'$, for $i\in[3]$ with primes $p_i>\sqrt{n}$. Note that $a_i'\neq b_i'$ for $i\in [3]$ as $a_i\neq b_i$. By how we construct $S$, this implies that $a_i'b_i'\in E(G)$ for all $i\in [3]$. However, we have $a_1'a_2'a_3'=b_1'b_2'b_3'$, or $\frac{a_1'}{b_1'}\cdot \frac{a_2'}{b_2'}=\frac{b_3'}{a_3'}$, which contradicts the fact that $G$ is product-free, proving the claim.

For $1\leq k\leq K$ and prime $p\in \left(\frac{n}{k+1},\frac{n}{k}\right]$, there are precisely $e(G_k)+k+1$ ways to include at most two multiples of $p$ as above. For different primes the  choices are independent, so, for sufficiently large $n$, the total number of sets that can be obtained in this way is at least
$$\prod\limits_{k=1}^K (e(G_k)+k+1)^{\pi(n/k)-\pi(n/(k+1))}>(\beta-\varepsilon)^{\pi(n)}.$$

\subsection{Upper bound}\label{sec-beta-theory-upper}
We now continue with the upper bound $S_3(n)\leq (\beta+o(1))^{\pi(n)}$.

Let $A\subseteq[n]$ be a multiplicative 3-Sidon set. We partition the elements of $A$ into three sets:
\begin{itemize}
	\item[] $A_1:=\{a\in A: \exists~ n^{2/3}<p \text{ prime s.t. } p\mid a\text{ and } p\text { divides at most 2 elements of $A$}\}$;
	
	\item[] $A_2:=\{a\in A: \exists~ n^{2/3}<p \text{ prime s.t. } p\mid a\text{ and } p \text{ divides at least 3 elements of $A$}\}$;

     \item[] $A_3:=\{a\in A: \text{ all prime divisors of }a\text{ are less than }n^{2/3} \}$.
\end{itemize}
Clearly, $A=A_1\cup A_2\cup A_3$.

We claim that both $A_2$ and $A_3$ are of small size, $|A_2\cup A_3|=n^{2/3+o(1)}$. Thus the contribution of $A_2\cup A_3$ to the number of multiplicative 3-Sidon sets is negligible: $n^{n^{2/3+o(1)}}=2^{o(\pi(n))}$. For $A_3$, this is already known, see e.g.~\cite{Pach18}, that $|A_3|=n^{2/3+o(1)}$.

For the set $A_2$, define the relevant set of primes
$$X:=\{p \text{ prime}> n^{2/3}:~p \text{ divides at least 3 elements of }  A_2\}.$$
Build an auxiliary bipartite graph $\Gamma$ with partite sets $X$ and $Y:=[n^{1/3}]$, in which two vertices form an edge in $\Gamma$ if their product is in $A_2$. As $p>n^{2/3}$, this means if $uv\in E(\Gamma)$, then $uv=a$ is the representation of some $a\in A_2$. So, we have
\begin{equation}\label{eq-eA2}
	|A_2|=e(\Gamma)\ge 3|X|.
\end{equation}
On the other hand, as $A_2$ is a multiplicative 3-Sidon set, it is not hard to check that $\Gamma$ is $C_6$-free. We may assume that $|X|\ge |Y|$, as otherwise $|A_2|\le |X||Y|\le n^{2/3}$. Then by Theorem~\ref{thm-c6-unbalanced} we have
$$|A_2|\le 2|X|+n^{2/3}/2.$$
Together with~\eqref{eq-eA2}, this implies that $|X|\le n^{2/3}/2$ and so
$$|A_2|\le 3n^{2/3}/2,$$
as claimed.

We are left to determine how many choices there are for $A_1$. For each large prime $p> n^{2/3}$, we can decide whether we add none/one/two of its multiples to $A_1$ (and which one(s)). Let $\eps>0$ be arbitrary and choose $K$ sufficiently large so that 
\begin{equation*}
	\prod\limits_{k\geq K} \left(\frac{k^2+k+2}{2}\right)^{\frac{1}{k^2+k}}<\frac{\beta+\varepsilon/4}{\beta+\eps/5}.
\end{equation*}
Then the contribution of multiples of   primes from $(n^{2/3},n/K)$ is at most 
$$\prod\limits_{k\geq K} \left({k\choose 2}+k+1\right)^{\pi(n/k)-\pi(n/(k+1))}<\left(\frac{\beta+\varepsilon/3}{\beta+\eps/5}\right)^{\pi(n)}.$$ 

We now bound the contribution of primes from $[n/K,n]$. We say that a pair $(a,b)$ \emph{witnesses} a prime $p\ge n/K$, if both $ap$ and $bp$ are in $A_1$. Assign to $A_1$ an auxiliary graph $G=(V,E)$ with
$$V:=[K], \quad\text{and }\quad E:=\{ab:~\exists~\text{distinct primes } p, q\ge n/K~\text{ s.t. } ap, bp, aq,bq \in A_1\},$$
that is, a pair $(a,b)$ form an edge in $G$ if it witnesses at least two large primes. For each edge $ab\in E(G)$, denote by $W(ab)$ the set of all primes witnessed by $(a,b)$. By the construction of $G$, we see that 
$$\forall ~e\in G,\quad |W(e)|\ge 2.$$ 
We call a prime $p\ge n/K$ \emph{irrelevant} (with respect to $G$) if 
\begin{itemize}
	\item[(i)] $p$ divides exactly two elements of $A_1$; and
	
	\item[(ii)] $p$ is not in any of the set $W(e)$, for $e\in E(G)$.
\end{itemize}
Then there are at most ${K\choose 2}$ irrelevant primes, contributing a factor of at most $n^{O(1)}$ to the choices of $A_1$.
For the \emph{relevant} primes, i.e. those either divides at most one element of $A_1$ or in $\cup_{e\in E(G)}W(e)$, observe that
\begin{itemize}
	\item[$(\ast)$] \emph{each prime $p\ge n/K$ can appear in at most one set $W(e)$ with $e\in E(G)$.}
\end{itemize}
Indeed, suppose to the contrary that a prime $p\ge n/K$ is in $W(e)\cap W(e')$ for two distinct edges $e, e'\in E$. By the definition of $G$, this means that $A_1$ contains at least $|e\cup e'|\ge 3$ multiples of $p$, contradicting the definition of $A_1$.

We claim that $G$ is product-free. Suppose there are three (not necessarily distinct) edges $e_i=a_ib_i$, $i\in[3]$ such that $\frac{a_1}{b_1}\cdot \frac{a_2}{b_2}=\frac{b_3}{a_3}$ or $a_1a_2a_3=b_1b_2b_3$. Then $\{e_1,e_2,e_3\}$ must contain at least two distinct edges. This, together with~$(\ast)$, implies that there exists distinct primes $p_i\ge n/K$, $i\in [3]$, such that $p_i\in W(e_i)$ or equivalently $a_ip_i, b_ip_i\in A_1$. Note that $a_ip_i, b_ip_i$, $i\in[3]$, are distinct, and
$$(a_1p_1)(a_2p_2)(a_3p_3)=(b_1p_1)(b_2p_2)(b_3p_3).$$ This contradicts $A_1$ being multiplicative 3-Sidon. 

The elements in $A_1$ with a relevant prime divisor larger than $n/K$ can now be constructed by first choosing a product-free graph $G$ on $K$ vertices, for which there are at most $2^{{K\choose 2}}$ choices; and then choosing for each prime $p\in (\frac{n}{k+1},\frac{n}{k}]$, $1\le k\le K$, at most two multiples according to $G_k$, for which there are at most $e(G_k)+k+1$ choices.

Recall, by the definition of $\beta$, that
$\prod\limits_{k=1}^K (e(G_k)+k+1)^{\frac{1}{k^2+k}}<\beta$.
Hence, the number of choices for $A_1$ is at most
$$\left(\frac{\beta+\varepsilon/3}{\beta+\eps/5}\right)^{\pi(n)}\cdot n^{O(1)}\cdot 2^{K\choose 2}\cdot \prod\limits_{k=1}^K (e(G_k)+k+1)^{\pi(n/k)-\pi(n/(k+1))}<(\beta+\eps/2)^{\pi(n)}.$$
Therefore, 
$$S_3(n)< n^{n^{2/3+o(1)}}\cdot  (\beta+\eps/2)^{\pi(n)} <(\beta+\eps)^{\pi(n)}$$
as desired.

\subsection{Estimating the limit $\beta$}\label{sec-beta-estimate}
Fix an $\eps>0$. Bounding the tail in $\beta$,
\begin{equation}\label{eq-tail}
	\prod_{k> K}(e(G_k)+k+1)\le \exp\left(\int_{K}^{\infty}\frac{2\log x}{x^2}dx\right)=\exp\left(\frac{2(\log K+1)}{K}\right)\le 1+\eps,
\end{equation}
we see that $\beta$ can be approximated up to a $(1+\eps)$ multiplicative error by searching for maximum-size product-free graphs on $K=\Theta(\frac{1}{\eps}\log \frac{1}{\eps})$ vertices. We now show a way to approximate $\beta$ avoiding finding a maximum product-free graph. In particular, we shall give upper and lower bounds that are within a ratio of $1.002$, showing that $\beta\approx 5.2$.

\subsubsection{Numerical bound from above}
For the numerical estimate, we will use the observation that every product-free graph $G$ is triangle-free. Indeed, as any triangle on vertices $a,b,c$ yields $\frac{a}{b}\cdot \frac{b}{c}=\frac{a}{c}$. Thus, by Mantel's Theorem, $e(G_k)\le \lfloor k^2/4 \rfloor$.
So $\beta\leq \beta^{+}$, where 
$$\beta^{+}:=\prod\limits_{k=1}^\infty \left(\lfloor k^2/4\rfloor+k+1\right)^{\frac{1}{k^2+k}}.$$
It is not hard to check that it is not possible to have $e(G_k)$ attaining the maximum size $\lfloor k^2/4\rfloor$ for every $k\le 10$, giving us an improvement 
$$\beta< 0.999744\beta^{+}.$$
We can bound $\beta^{+}$ by its partial product up to some large $K$ and estimate its tail using~\eqref{eq-tail}: 
$$\beta^{+}\le \prod\limits_{k=1}^{K} \left(\lfloor k^2/4\rfloor+k+1\right)^{\frac{1}{k^2+k}}\cdot \exp\left(\frac{2(\log K+1)}{K}\right).$$
By taking $K=30000$, we then get an upper estimate
$\beta< 5.2468$.

\subsubsection{Numerical bound from below}
We shall construct a bipartite product-free graph that gets ``quite'' close to the maximum size. Partition $\N$ into two classes $N_0$ and $N_1$ according to the parity of $\Omega(x)$, the number of prime divisors with multiplicity:
\begin{itemize}
	\item[] $N_0:=\{x\in\N: \Omega(x)\equiv 0\pmod{2}   \}$;

    \item[] $N_1:=\{x\in\N: \Omega(x)\equiv 1\pmod{2}   \}$.
\end{itemize}
Let $G_{\text{par}}$ be the bipartite graph on $\N$ with partite sets $N_0$ and $N_1$. By construction, we have that
\begin{equation}\label{eq-parity}
      \forall~ ab\in E(G_{\text{par}}), \quad  2\nmid \Omega(ab).
\end{equation}
Suppose there are three (not necessarily distinct) edges $a_ib_i$, $i\in[3]$ such that $\frac{a_1}{b_1}\cdot \frac{a_2}{b_2}=\frac{b_3}{a_3}$ or $a_1a_2a_3=b_1b_2b_3$. Then $\Omega(a_1a_2a_3)=\Omega(b_1b_2b_3)$, and, as $\Omega(\cdot)$ is completely additive, we have $2\mid \Omega(\prod_{i\in[3]}a_ib_i)$, contradicting~\eqref{eq-parity}. Thus $G_{\text{par}}$ is product-free. Recall that there are exactly $ (k+L(k))/2$ and $(k-L(k))/2$ elements in $[k]$ with even and odd number of divisors with multiplicity respectively, where $L(k)$ is the summatory Liouville function. We then have
$\beta^{-}\leq \beta$ for 
\begin{equation}\label{eq-c1}
	\beta^{-}:= \prod\limits_{k=1}^\infty \left(k^2/4-L(k)^2/4+k+1\right)^{\frac{1}{k^2+k}}.
\end{equation}  
We remark that there are infinitely many identical terms in the products in $\beta^{-}$ and $\beta^{+}$ as the summatory Liouville function takes value zero infinitely often.

We can bound $\beta^{-}$ from below by its partial product up to $K=30000$ and obtain $\beta>5.2366$. Thus, the ratio of the upper and lower estimates is less than $1.002$.

\section{Generalised multiplicative Sidon set}\label{sec-k-sidon}
In this section, we sketch the proof of Theorem~\ref{thm-k-sidon}. We start with the following simple but useful observation. Consider a multiplicative $4$-Sidon set $A$. Fix (if exists) a 4-tuple $B$ in $A$ satisfying the equation $a_1a_2=b_1b_2$. Then $A\setminus B$ must be multiplicative $2$-Sidon. Thus, 
$$S_4(n)\le S_2(n)+S_2(n)\cdot{n\choose 4}.$$ 
In general, we have for all $k\ge 2$ that
$$(\dagger)\quad \quad\quad \quad \quad S_k(n)\le \left\{ 
\begin{array}{ll}
S_2(n)\cdot \sum_{i=0}^{\frac{k}{2}-1}{n\choose 4i}, & \mbox{if $k$ is even};\\
S_2(n)\cdot \sum_{i=0}^{\frac{k-5}{2}}{n\choose 4i}+S_3(n){n\choose 2(k-3)}, & \mbox{if $k\ge 3$ is odd}.\end{array} \right.$$
As any set consisting of at most one multiple of each prime larger than $n^{2/3}$ is multiplicative $k$-Sidon for all $k\ge 2$, we see that $T(n)$ is also a lower bound for $S_k(n)$. This shows that $\log S_k(n)$ are asymptotically the same for all even $k\ge 2$.

The proof of Theorem~\ref{thm-k-sidon} for odd $k\ge 5$ is very similar to that of Theorem~\ref{thm-count-3-sidon}. We highlight here only the differences.

For a graph $G$, define 
$$R(G):=\{r:~ r=a/b, ab\in E(G)\}$$ 
the set of all ratios of edges in $G$. Generalising the notion of product-free graphs, we say that a graph $G$ is \emph{$k$-product-free} if $R(G)$ does not contain any solution to the equation $x_1x_2\ldots x_{k-1}=x_k$. Note that here we do \emph{not} require the $x_i$s in the solution to be distinct. Writing $\cG_k$ for the family of all $k$-product-free graphs, define analogously
$$\beta_k:=\sup_{G\in\cG_k}~\prod\limits_{k=1}^\infty (e(G_k)+k+1)^{\frac{1}{k^2+k}}.$$
Note that for any odd $k', k$ with $k'>k$, as we can add pairs of reciprocal ratios from $R(G)$ to the left-hand-side of $x_1x_2\ldots x_{k-1}=x_k$ to get a solution for $x_1x_2\ldots x_{k'-1}=x_{k'}$, we see that   
$$\cG\supseteq \cG_5\supseteq \cG_7\supseteq \ldots$$ 
is a nested sequence.
Then, as in Section~\ref{sec-beta-theory-lower}, we have  
$S_k(n)\ge (\beta_k-o(1))^{\pi(n)}$. 

To bound $S_k(n)$ from above, for a multiplicative $k$-Sidon set $A$, define the sets $A_1,A_2,A_3$ and the graph $\Gamma$ on $X\cup Y$ with $|X|\ge |Y|$ exactly as in Section~\ref{sec-beta-theory-upper}. Then again $|A_3|=n^{2/3+o(1)}$~\cite{Pach15}. For $A_2$, we still have $|A_2|\ge 3|X|$. Recall that each edge in $\Gamma$ corresponds to an element in $A_2$. 

The new idea we need here to bound $A_2$ is that if $A_2$ is somewhat larger than $n^{2/3}$, then we can find edge-disjoint cycles with one copy of $C_6$ and $(k-3)/2$ copies of  $C_4$'s. Then the elements in $A_2$ corresponding to the edges in these cycles are all distinct and form a solution to $a_1a_2\ldots a_k=b_1b_2\ldots b_k$, giving us a contradiction. More precisely, suppose that $\Gamma$ is $C_4$-free, then by Theorem~\ref{thm-c4-unbalanced}, we have 
$$3|X|\le |A_2|=e(\Gamma)\le |X|^{1/2}|Y|+|X|,$$ 
implying that $|X|\le |Y|^2/4$. Consequently, $|A_2|\le 3|Y|^2/4=3n^{2/3}/4$. We may thus assume that $\Gamma$ contains a copy of $C_4$, call it $F_1$. Let $\Gamma_1$ be the graph obtained from $\Gamma$ by deleting the edges in $F_1$: $\Gamma_1:=\Gamma\setminus E(F_1)$. So $e(\Gamma_1)\ge e(\Gamma)-4\ge 3|X|-4$. Then again we see that either $A_2$ is of size $O(n^{2/3})$ or there exists a copy of $C_4$, say $F_2$, in $\Gamma_1$. Define then $\Gamma_2:=\Gamma_1\setminus E(F_2)$. We repeat this process $\frac{k-3}{2}$ times to obtain $\Gamma_0:=\Gamma_{\frac{k-3}{2}}$ and pairwise edge-disjoint $F_i\subseteq \Gamma$, $i\in[\frac{k-3}{2}]$, each isomorphic to $C_4$. We claim that $\Gamma_0$ is $C_6$-free. Indeed, a copy of $C_6$ in $\Gamma_0$ together with the pairwise edge-disjoint $(k-3)/2$ copies of $C_4$'s we have found would yield a solution to $a_1a_2\ldots a_k=b_1b_2\ldots b_k$. Thus, we have by Theorem~\ref{thm-c6-unbalanced} that
$$3|X|-2(k-3)\le e(\Gamma_0)\le 2|X|+|Y|^2/2,$$
and so $|X|< n^{2/3}$. This implies 
$$|A_2|=e(\Gamma)\le e(\Gamma_0)+2(k-3)\le 3n^{2/3}.$$
Thus, the main contribution to $S_k(n)$ comes again from the number of choices for $A_1$.

For $A_1$, we shall define an auxiliary graph $G$ on vertex set $[K]$, and let $ab\in E(G)$ if $(a,b)$ witnesses at least $k$ primes larger than $n/K$. Then note that $G$ is now $k$-product-free, and we can similarly obtain the upper bound $(\beta_k+o(1))^{\pi(n)}$ for the number of choices for $A_1$, hence also for $S_k(n)$.

Note that the bipartite product-free graph $G_{\text{par}}$ in Section~\ref{sec-beta-estimate} is in fact in $\cG_k$ for all odd $k\ge 3$ and the corresponding construction yields multiplicative $k$-Sidon sets. Thus for all odd $k\ge 3$, 
$$\beta_k\ge \beta^{-}.$$
Both~$(\dagger)$ and that $\{\cG_k\}$, $k\ge 3$ odd, is a nested sequence imply that the sequence 
$$\beta\ge \beta_5\ge \beta_7\ge\ldots \ge \beta^{-}$$ 
is non-increasing.

\section{Concluding remarks}\label{sec-conclude}
In this paper, we determine the number of multiplicative Sidon subsets of $[n]$, giving bounds that are optimal up to a constant factor in the exponent of the lower order term $2^{\Theta\left(\frac{n^{3/4}}{(\log n)^{3/2}}\right)}$. For generalised multiplicative Sidon sets, we show that for even $k\ge 2$, $\log S_k(n)$ are asymptotically the same; while for odd $k\ge 3$, the limit 
$$\beta_k=\lim_{n\rightarrow \infty}S_k(n)^{1/\pi(n)}$$ 
exists, and the limits form a non-increasing sequence 
$$\beta\ge\beta_5\ge \beta_7\ge\ldots\ge \beta^{-}.$$ 

When approximating $\beta$ from below, we constructed a bipartite product-free graph $G_{\text{par}}$ using the parity of $\Omega(x)$, the number of prime divisors of $x$ with multiplicity. We conjecture that this lower estimate from the Liouville-type constant $\beta^{-}$ in~\eqref{eq-c1} provides the correct value of $\beta$, i.e. all equalities hold above.
In other words, $G_{\text{par}}$ realises the supremum in~\eqref{eq-beta} and $\log S_k(n)$ are asymptotically the same for all odd $k\ge 3$.

\bigskip


\end{document}